\documentclass[11pt]{amsart}
\usepackage{amsmath,amssymb,xy,array}
\usepackage[T1]{fontenc}
\usepackage{amsfonts}
\usepackage{amsmath}
\usepackage{amssymb}
\usepackage{amsthm}
\usepackage[cp850]{inputenc}
\usepackage{hyperref}
\usepackage{graphicx}
\usepackage{fancyhdr}
\usepackage{tikz}
\usepackage{longtable}
\usepackage{geometry}
\usepackage{textcomp}
\usetikzlibrary{trees}

\newcommand{\arXiv}[1]{\href{http://arxiv.org/abs/#1}{arXiv:#1}}
\def\bibaut#1{{\sc #1}}

\providecommand{\U}[1]{\protect\rule{.1in}{.1in}}
\providecommand{\U}[1]{\protect\rule{.1in}{.1in}}
\providecommand{\U}[1]{\protect\rule{.1in}{.1in}}
\providecommand{\U}[1]{\protect\rule{.1in}{.1in}}
\providecommand{\U}[1]{\protect\rule{.1in}{.1in}}
\input{xy}
\xyoption{all}
\usepackage{amsfonts}
\usepackage{amssymb,amsthm}
\usepackage{verbatim}
\usepackage{hyperref} 
\usepackage[T1]{fontenc}
\usepackage{amsmath}
\usepackage{hyperref}
\usepackage{enumerate}
\usepackage{tikz}
\usepackage{color}

\newcommand{\C}{\mathbb C}
\newcommand{\G}{\mathbb G}
\renewcommand{\P}{\mathbb P}

\newcommand{\Z}{\mathbb Z}

\DeclareMathOperator{\Sec}{Sec}

\newcommand{\QED}{\ifhmode\unskip\nobreak\fi\quad {\rm Q.E.D.}} 

\newcommand{\N}{\mathbb{N}}
 
\renewcommand{\P}{\mathbb{P}}

\renewcommand{\sec}{\mathbb{S}ec}
\DeclareMathOperator{\expdim}{expdim}
\DeclareMathOperator{\Sym}{Sym}

\newtheorem{Theorem}{Theorem}[section]

\newtheorem{Lemma}[Theorem]{Lemma}
\newtheorem{Proposition}[Theorem]{Proposition}

\newtheorem{Corollary}[Theorem]{Corollary}

\theoremstyle{definition}

\newtheorem{Definition}[Theorem]{Definition}
\newtheorem{Remark}[Theorem]{Remark}

\newtheorem{Example}[Theorem]{Example}
\newtheorem{Notation}[Theorem]{Notation}

\hypersetup{pdfpagemode=UseNone}
\hypersetup{pdfstartview=FitH}
				
\begin{document}
\title{On non-secant defectivity of Segre-Veronese varieties}

\author[Carolina Araujo]{Carolina Araujo}
\address{\sc Carolina Araujo\\ IMPA, Estrada Dona Castorina 110, Rio de Janeiro, Brazil}
\email{caraujo@impa.br}

\author[Alex Massarenti]{Alex Massarenti}
\address{\sc Alex Massarenti\\
Universidade Federal Fluminense\\
Rua M\'ario Santos Braga\\
24020-140, Niter\'oi,  Rio de Janeiro\\ Brazil}
\email{alexmassarenti@id.uff.br}

\author[Rick Rischter]{Rick Rischter}
\address{\sc Rick Rischter\\ Universidade Federal de Itajub\'a, Av. BPS 1303, Bairro Pinheirinho, Itajub\'a, Minas Gerais, Brazil}
\email{rischter@unifei.edu.br}

\date{\today}

\maketitle

\begin{abstract}
Let $SV^{\pmb n}_{\pmb d}$ be the Segre-Veronese given as the image of the embedding induced by the line bundle $\mathcal{O}_{\mathbb{P}^{n_1}\times\dots\times\mathbb{P}^{n_r}}(d_1,\dots, d_r)$. We prove that asymptotically $SV^{\pmb n}_{\pmb d}$ is not $h$-defective for $h\leq n_1^{\lfloor \log_2(d-1)\rfloor}$.
\end{abstract}

\setcounter{tocdepth}{1}
\tableofcontents

%
%

\section{Introduction}

Secant varieties are classical objects in algebraic geometry.
The \textit{$h$-secant variety} $\sec_{h}(X)$ of a non-degenerate $n$-dimensional variety $X\subset\mathbb{P}^N$ is the Zariski closure 
of the union of all linear spaces spanned by collections of $h$ points of $X$. 
The \textit{expected dimension} of $\sec_{h}(X)$ is
$$
\expdim(\sec_{h}(X)):= \min\{nh+h-1,N\}.
$$
The actual dimension of $\sec_{h}(X)$ may be smaller than the expected one. This happens when there are infinitely many $(h-1)$-planes 
$h$-secant to $X$ passing trough a general point of  $\sec_{h}(X)$. 
Following \cite{Za}, we say that $X$ is \textit{$h$-defective} if 
$$
\dim(\sec_{h}(X)) < \expdim(\sec_{h}(X)).
$$
Determining secant defectivity is an old problem  in algebraic geometry, which goes back to the Italian school 
(see  \cite{Ca37}, \cite{Sc08}, \cite{Se01}, \cite{Te11}).

In this paper we investigate secant defectivity for  Segre-Veronese varieties. 
The problem is specially interesting in this case, in connection with problems of tensor decomposition 
(see  \cite{CM}, \cite{CGLM}, \cite{La12}).
Indeed, Segre-Veronese varieties parametrize rank one tensors. 
So their $h$-secant varieties parametrize tensors of a given rank depending on $h$. 
For this reason, they have been used to construct and study moduli spaces for additive decompositions of a general tensor into 
a given number of rank one tensors (see \cite{Do04}, \cite{DK93}, \cite{Ma16}, \cite{MM13}, \cite{RS00}, \cite{TZ11}, \cite{BGI11}).

The problem of secant defectivity for  Veronese varieties was completely solved in \cite{AH95}. 
In that paper, Alexander and  Hirshowitz showed that, except for the degree $2$ Veronese embedding, which is almost always defective, 
the degree $d$ Veronese embedding of $\mathbb{P}^n$ is not $h$-defective except in the following cases: 
$$
(d,n,h)\in\{(4,2,5),(4,3,9),(3,4,7),(4,4,14)\}.
$$

For Segre varieties, secant defectivity in classified in some special cases. 
Segre products of two factors $\mathbb{P}^{n_1}\times\mathbb{P}^{n_2}\subset \mathbb{P}^{^{n_1n_2+n_1+n_2}}$ are almost always defective.
For Segre products  $\mathbb{P}^{1}\times\dots\times\mathbb{P}^{1}\subset\mathbb{P}^N$, the problem was completely 
settled in \cite{CGG11}.
In general, $h$-defectivity of Segre products $\mathbb{P}^{n_1}\times\dots\times\mathbb{P}^{n_r}\subset\mathbb{P}^N$ is classified only for $h\leq 6$
(\cite{AOP09}).

Next we turn to Segre-Veronese varieties. These are products $\mathbb{P}^{n_1}\times\dots\times\mathbb{P}^{n_r}$ embedded by the complete 
linear system  $\big|\mathcal{O}_{\mathbb{P}^{n_1}\times\dots\times\mathbb{P}^{n_r}}(d_1,\dots, d_r)\big|$, $d_i>0$.
The problem  of secant defectivity  for  Segre-Veronese varieties has been solved in some very special cases, mostly for products of few factors 
(see \cite{CGG03}, \cite{AB09}, \cite{Ab10}, \cite{BCC11}, \cite{AB12}, \cite{BBC12}, \cite{AB13}).
Secant defectivity for  Segre-Veronese products  $\mathbb{P}^{1}\times\dots\times\mathbb{P}^{1}$, with arbitrary number of factor and degrees, 
was classified in \cite{LP13}. 
In general, $h$-defectivity is classified only for small values of $h$ (\cite[Proposition 3.2]{CGG03}):
except for the Segre product $\mathbb{P}^{1}\times\mathbb{P}^{1}\subset\mathbb{P}^3$, Segre-Veronese varieties 
$\mathbb{P}^{n_1}\times\dots\times\mathbb{P}^{n_r}$ are never $h$-defective for $h\leq \min\{n_i\}+1$.
In this paper we improve this bound by taking into account the embedding degrees $d_1,\dots, d_r$.
We show that, asymptotically, Segre-Veronese varieties are never $h$-defective for $h\leq (\min\{n_i\})^{\lfloor \log_2(d-1)\rfloor}$,
where $d=d_1+\dots+ d_r$.
More precisely, our main result in Theorem~\ref{main} can be rephrased as follows.

\begin{Theorem}\label{main_intro}
Let $\pmb{n}=(n_1,\dots,n_r)$ and $\pmb{d} = (d_1,\dots,d_r)$ be two $r$-uples of positive integers, with $n_1\leq \dots \leq n_r$ and $d=d_1+\dots+d_r\geq 3$. 
Let $SV^{\pmb n}_{\pmb d}\subset\mathbb{P}^N$ be the product  $\mathbb{P}^{n_1}\times\dots\times\mathbb{P}^{n_r}$ embedded by the complete 
linear system  $\big|\mathcal{O}_{\mathbb{P}^{n_1}\times\dots\times\mathbb{P}^{n_r}}(d_1,\dots, d_r)\big|$. 
Write 
$$
d-1 = 2^{\lambda_1}+\dots+2^{\lambda_s}+\epsilon,
$$ 
with $\lambda_1 > \lambda_2 >\dots >\lambda_s\geq 1$, $\epsilon\in\{0,1\}$. 
Then $SV^{\pmb n}_{\pmb d}$ is not $(h+1)$-defective for 
$$
h\leq n_1((n_1+1)^{\lambda_1-1}+\dots + (n_1+1)^{\lambda_s-1})+1.
$$
\end{Theorem}

Our proof of Theorem~\ref{main_intro} follows the strategy introduced in \cite{MRi16}, which we now explain. 
Given a non-degenerate $n$-dimensional variety $X\subset\mathbb{P}^N$, and general points  $x_1,\dots,x_h\in X\subset\mathbb{P}^N$, consider the 
linear projection with center $\left\langle T_{x_1}X,\dots,T_{x_h}X\right\rangle$,
$$
\tau_{X,h}:X\subseteq\mathbb{P}^N\dasharrow\mathbb{P}^{N_h}.
$$
By \cite[Proposition 3.5]{CC01}, if $\tau_{X,h}$ is generically finite then $X$ is not $(h+1)$-defective.
In general, however, it is hard to control the dimension of the fibers of the tangential projections $\tau_{X,h}$ as $h$ gets larger. 
In \cite{MRi16} a new strategy  was developed, based on the more general  \emph{osculating projections} instead of just tangential projections.
For a smooth point $x\in X\subset\mathbb{P}^N$, the \textit{$k$-osculating space} $T_x^{k}X$ of $X$ at $x$ is roughly the smaller linear subspace 
where $X$ can be locally approximated up to order $k$ at $x$ (see Definition~\ref{oscdef}).
Given $x_1,\dots, x_l\in X$ general points, we consider the linear projection with center $\left\langle T_{x_1}^{k_1}X,\dots, T_{x_l}^{k_l}X\right\rangle$,
$$
\Pi_{T^{k_1,\dots,k_l}_{x_1,\dots,x_l}}:X\subset\mathbb{P}^N\dasharrow\mathbb{P}^{N_{k_1,\dots,k_l}},
$$
and call it a \textit{$(k_1+\dots +k_l)$-osculating projection}.
Under suitable conditions, one can degenerate the linear span of several tangent spaces $T_{x_i}X$ into a subspace contained in a
single osculating space $T_x^{k}X$. 
So the tangential projections $\tau_{X,h}$ degenerates to a linear projection with center contained in the linear span of osculating spaces,
$\left\langle T_{p_1}^{k}X,\dots, T_{p_l}^{k}X\right\rangle$.
If $\Pi_{T^{k,\dots,k}_{p_1,\dots,p_l}}$ is generically finite, then $\tau_{X,h}$ is also generically finite, and one concludes that $X$ is not $(h+1)$-defective.
The advantage of this approach is that one has to consider osculating spaces at much less points than $h$, 
allowing to control the dimension of the fibers of the projection. 
In  \cite{MRi16}, this strategy was successfully applied to study the problem of secant defectivity for Grassmannians. 
Here we apply it to Segre-Veronese varieties.

The paper is organized as follows. 
In Section \ref{oscspaces} we describe explicitly osculating spaces of Segre-Veronese varieties.
In Section \ref{oscproj} we study the relative dimension of general osculating projections. 
In Section \ref{degtanosc} we study how many general tangent projections degenerate to osculating projections. 
Finally, in Section \ref{mainsec} we apply these result and the techniques developed in \cite{MRi16}  
to prove our main result on the dimension of secant varieties of Segre-Veronese varieties.

\

\noindent {\bf Notation and conventions.}
We always work over the field ${\mathbb C}$ of complex numbers. 
Varieties are always assumed to be irreducible. 
For a vector space $V$, we denote by $\mathbb{P}(V)$ its Grothendieck projectivization, i.e., the projective space of
nonzero linear forms on $V$ up to scaling.

\

\noindent {\bf Acknowledgments.}
We would like to thank Maria Chiara Brambilla for a detailed account on the state of the art about secant defectivity of Segre-Veronese varieties.\\
The first named author was partially supported by CNPq and Faperj Research Fellowships. The second named author is a member of the Gruppo Nazionale per le Strutture Algebriche, Geometriche e le loro Applicazioni of the Istituto Nazionale di Alta Matematica "F. Severi" (GNSAGA-INDAM). The third named author would like to thanks CNPq for the financial support.

%
%

\section{Osculating spaces of Segre-Veronese varieties}\label{oscspaces}

In this section we describe osculating spaces of Segre-Veronese varieties.
We start by defining osculating spaces. 
They can also be defined intrinsically via jet bundles (see \cite[Section 3]{MRi16}). 
 
\begin{Definition}\label{oscdef}
Let $X\subset \P^N$ be a projective variety of dimension $n$, and $p\in X$ a smooth point.
Choose a local parametrization of $X$ at $p$:
$$
\begin{array}{cccc}
\phi: &\mathcal{U}\subseteq\mathbb{C}^n& \longrightarrow & \mathbb{C}^{N}\\
      & (t_1,\dots,t_n) & \longmapsto & \phi(t_1,\dots,t_n) \\
      & 0 & \longmapsto & p 
\end{array}
$$
For a multi-index $I = (i_1,\dots,i_n)$, set 
\begin{equation}\label{osceq}
\phi_I = \frac{\partial^{|I|}\phi}{\partial t_1^{i_1}\dots\partial t_n^{i_n}}.
\end{equation}
For any $m\geq 0$, let $O^m_pX$ be the affine subspace of $\mathbb{C}^{N}$ centered at $p$ and spanned by the vectors $\phi_I(0)$ with  $|I|\leq m$.

The $m$-\textit{osculating space} $T_p^m X$ of $X$ at $p$ is the projective closure  of  $O^m_pX$ in $\mathbb{P}^N$.
Note that $T_p^0 X=\{p\}$, and $T_p^1 X$ is the usual tangent space of $X$ at $p$.


The \textit{$m$-osculating dimension} of  $X$ at $p$ is 
\begin{equation*}
\dim(T_p^m X) = \binom{n+m}{n}-1-\delta_{m,p},
\end{equation*}
where $\delta_{m,p}$ is the number of independent differential equations of order $\leq m$ satisfied by $X$ at $p$. 
\end{Definition}

Next we turn to Segre-Veronese varieties.
In Proposition~\ref{oscsegver} below, 
we describe explicitly their osculating spaces at \emph{coordinate points} by computing \eqref{osceq} for 
a suitable rational parametrization.  
In order to do so, we recall the definition of Segre-Veronese varieties and fix some notation to be used throughout the paper.

\begin{Notation}\label{notationSV}
Let $\pmb{n}=(n_1,\dots,n_r)$ and $\pmb{d} = (d_1,\dots,d_r)$ be two $r$-uples of positive integers, with $n_1\leq \dots \leq n_r$.
Set $d=d_1+\dots+d_r$,  $n=n_1+\dots+n_r$, and $N(\pmb{n},\pmb{d})=\prod_{i=1}^r\binom{n_i+d_i}{n_i}-1$. 

Let $V_1,\dots, V_r$ be vector spaces of dimensions $n_1+1\leq n_2+1\leq \dots \leq n_r+1$, and consider the product
$$
\mathbb{P}^{\pmb{n}} = \mathbb{P}(V_1^{*})\times \dots \times \mathbb{P}(V_r^{*}).
$$
The line bundle 
$$
\mathcal{O}_{\mathbb{P}^{\pmb{n}} }(d_1,\dots, d_r)=\mathcal{O}_{\mathbb{P}(V_1^{*})}(d_1)\boxtimes\dots\boxtimes \mathcal{O}_{\mathbb{P}(V_1^{*})}(d_r)
$$
induces an embedding
$$
\begin{array}{cccc}
\sigma\nu_{\pmb{d}}^{\pmb{n}}:
&\mathbb{P}(V_1^{*})\times \dots \times \mathbb{P}(V_r^{*})& \longrightarrow &
\mathbb{P}(\Sym^{d_1}V_1^{*}\otimes\dots\otimes \Sym^{d_r}V_r^{*})
=\mathbb{P}^{N(\pmb{n},\pmb{d})},\\
      & (\left[v_1\right],\dots,\left[v_r\right]) & \longmapsto & [v_1^{d_1}\otimes\dots\otimes v_r^{d_r}]
\end{array}
$$ 
where $v_i\in V_i$.
We call the image 
$$
SV_{\pmb{d}}^{\pmb{n}}= \sigma\nu_{\pmb{d}}^{\pmb{n}}(\mathbb{P}^{\pmb{n}} ) \subset \mathbb{P}^{N(\pmb{n},\pmb{d})}
$$ 
a \textit{Segre-Veronese variety}. It is a smooth  variety of dimension $n$ and degree 
$\frac{(n_1+\dots +n_r)!}{n_1!\dots n_r!}d_1^{n_1}\dots d_r^{n_r}$ in $\mathbb{P}^{N(\pmb{n},\pmb{d})}$. 

When $r = 1$, $SV_{d}^{n}$ is a Veronese variety. In this case we write $V_d^n$ for $SV_{d}^{n}$, and $v_{d}^{n}$ for the Veronese embedding.
When $d_1 = \dots = d_r = 1$, $SV_{1,\dots,1}^{\pmb{n}}$ is a Segre variety. 
In this case we write $S^{\pmb{n}}$ for $SV_{1,\dots,1}^{\pmb{n}}$, and  $\sigma^{\pmb{n}}$ for the Segre embedding.
Note that 
$$
\sigma\nu_{\pmb{d}}^{\pmb{n}}=\sigma^{\pmb{n}'}\circ \left( \nu_{d_1}^{n_1}\times \dots \times \nu_{d_r}^{n_r}\right),
$$
where $\pmb{n}'=(N(n_1,d_1),\dots,N(n_r,d_r))$.

Given $v_1, \dots, v_{d_j}\in V_j$, we denote by $v_1\cdot \dots \cdot v_{d_j} \in \Sym^{d_j}V_j$  the symmetrization of
$v_1\otimes \dots \otimes v_{d_j}$.

Hoping that no confusion will arrive, we write $(e_0,\dots,e_{n_j})$ for a fixed a basis of each $V_j$. 
Given a $d_j$-uple $I=(i_1,\dots,i_{d_j})$, with $0\leq i_1 \leq \dots \leq i_{d_j} \leq n_j$, we denote by 
$e_{I}\in \Sym^{d_j}V_j$  the symmetric product $e_{i_{1}}\cdot\dots\cdot e_{i_{d_j}}$.

For each $j\in\{1, \dots, r\}$, consider a $d_j$-uple $I^j=(i^j_1,\dots,i^j_{d_j})$, with $0\leq i^j_1 \leq \dots \leq i^j_{d_j} \leq n_j$, and
set 
\begin{equation*}\label{vector}
I = (I^1,\dots,I^r)=((i_{1}^{1},\dots,i_{d_1}^{1}),
(i_{1}^{2},\dots,i_{d_2}^{2}),\dots,(i_{1}^{r},\dots,i_{d_r}^{r})).
\end{equation*}
We denote by $e_I$ the vector
$$
e_I = e_{I^1}\otimes e_{I^2} \otimes \cdots \otimes e_{I^r}\in \Sym^{d_1}V_1\otimes\dots\otimes \Sym^{d_r}V_r,
$$
as well as the corresponding point in 
$\mathbb{P}(\Sym^{d_1}V_1^{*}\otimes\dots\otimes \Sym^{d_r}V_r^{*})=\mathbb{P}^{N(\pmb{n},\pmb{d})}$.
When $I^j=(i_j,\dots,i_j)$ for every $j\in\{1, \dots, r\}$, for some $0\leq i_j \leq n_j$, we have
$$
e_I=\sigma\nu_{\pmb{d}}^{\pmb{n}}(\left[e_{i_1}\right],\dots,\left[e_{i_r}\right])
\in SV_{\pmb{d}}^{\pmb{n}}\subset \mathbb{P}^{N(\pmb{n},\pmb{d})}.
$$
In this case we say that $e_I$ is a \emph{coordinate point} of $SV_{\pmb{d}}^{\pmb{n}}$.
\end{Notation}

\begin{Definition}\label{distance}
Let $n$ and $d$ be positive integers, and set 
$$
\Lambda_{n,d}=\{I=(i_1,\dots,i_{d}),0\leq i_1 \leq \dots \leq i_{d} \leq n\}.
$$
For $I,J\in \Lambda_{n,d}$, we define their distance $d(I,J)$ as the number of different coordinates.
More precisely, write $I=(i_1,\dots,i_{d})$ and $J=(j_1,\dots,j_{d})$. 
There are $r\geq 0$, distinct indices $\lambda_1, \dots, \lambda_r\subset \{1, \dots, d\}$, and 
distinct indices $\tau_1, \dots, \tau_r\subset \{1, \dots, d\}$ such that $i_{\lambda_k}=j_{\tau_k}$ for every $1\leq k\leq r$, and 
$\{i_\lambda | \lambda\neq \lambda_1, \dots, \lambda_r\}\cap \{j_\tau | \tau\neq \tau_1, \dots, \tau_r\}=\emptyset$.
Then $d(I,J)=d-r$.
Note that $\Lambda_{n,d}$ has diameter $d$ and size $\binom{n+d}{n}=N(n,d)+1$.

Let $\pmb{n}=(n_1,\dots,n_r)$ and $\pmb{d} = (d_1,\dots,d_r)$ be two $r$-uples of positive integers, and set 
$$
\Lambda=\Lambda_{\pmb{n},\pmb{d}}=\Lambda_{n_1,d_1}\times \dots \times \Lambda_{n_r,d_r}.
$$
For $I=(I^1,\dots,I^r),J=(J^1,\dots,J^r)\in \Lambda$, 
we define their distance as 
$$
d(I,J)=d(I^1,J^1)+\dots+d(I^r,J^r).
$$
Note that $\Lambda$ has diameter $d$ and size $\prod_{i=1}^r\binom{n_i+d_i}{n_i}=N(\pmb{n},\pmb{d})+1$.
\end{Definition}

We can now state the main result of this section. 

\begin{Proposition}\label{oscsegver}
Let the notation and assumptions be as in Notation~\ref{notationSV} and Definition~\ref{distance}.
Set $I^1=(i_1,\dots,i_1),\dots, I^r=(i_r,\dots,i_r)$, with $0\leq i_j \leq n_j$, and $I = (I^1,\dots,I^r)$. 
Consider the point 
$$
e_I=\sigma\nu_{\pmb{d}}^{\pmb{n}}(\left[e_{i_1}\right],\dots,\left[e_{i_r}\right]) \in SV_{\pmb{d}}^{\pmb{n}}.
$$
For any $s\geq 0$, we have 
$$
T^s_{e_I}(SV_{\pmb{d}}^{\pmb{n}})= \left\langle e_J \: | \: d(I,J)\leq s\right\rangle. 
$$
In particular, $T^s_{e_I}(SV_{\pmb{d}}^{\pmb{n}})=\P^{N(\pmb{n},\pmb{d})}$ for any $s\geq d.$
\end{Proposition}

\begin{proof}
We may assume that $I^1=(0,\dots,0),\dots,I^r=(0,\dots,0)$. 
Write $\big(z_K\big)_{K\in \Lambda}$, for coordinates in $\P^{N(\pmb{n},\pmb{d})}$, and 
consider the rational parametrization
$$
\phi:\mathbb{A}^{\prod n_i}\rightarrow SV_{\pmb{d}}^{\pmb{n}}\cap \big(z_I\neq 0\big)\subset \mathbb{A}^{N(\pmb{n},\pmb{d})}
$$
given by 
$$
A=(a_{j,i})_{j=1,\dots, r,i=1,\dots,n_j}
\mapsto \left( \prod_{j=1}^r \prod_{k=1}^{d_j} a_{j,i_{k}^{j}}\right)_{K=(K^1,\dots,K^r)\in \Lambda \backslash \{I\}},
$$
where $K^{j}=(i_{1}^j,\dots,i_{d_j}^j)\in \Lambda_{n_j,d_j}$ for each $j=1,...,r$.

For integers $l$ and $m$, we write $\deg_{l,m}K$ for the degree of the polynomial
$$
\phi(A)_{K}:=\prod_{j=1}^r \prod_{k=1}^{d_j} a_{j,i_{k}^{j}}
$$
with respect to $a_{l,m}$. Then $0\leq \deg_{l,m}K\leq d_l$, and the degree of $\phi(A)_{K}$
with respect to all the variables $a_{j,i}$ is at most $d$. 
One computes:
$$
\left(\dfrac{\partial^{\lambda_1+\dots+\lambda_t} \phi(A)}
{\partial^{\lambda_1}a_{l_1,m_1}\dots\partial^{\lambda_t}a_{l_t,m_t}}\right)_{K}\!\!\!=\!
\begin{cases}
0 &\mbox{if } \deg_{l_j,m_j} K<\lambda_j \mbox{ for some } j.\\
\dfrac{\prod_{j=1}^t (\deg_{l_j,m_j}K) !\phi(A)_{K}}
{\prod_{j=1}^t (\deg_{l_j,m_j}K-\lambda_j) ! a_{l_j,m_j}^{\lambda_j}} &\mbox{otherwise }
\end{cases}
$$
For $A = 0$ we get
$$
\left(\dfrac{\partial^{\lambda_1+\dots+\lambda_t} \phi(0)}
{\partial^{\lambda_1}a_{l_1,m_1}\dots\partial^{\lambda_t}a_{l_t,m_t}}\right)_{K}\!\!=
\begin{cases}
0 &\mbox{if } \deg_{l_j,m_j} K\neq\lambda_j \mbox{ for some } j.\\
\prod_{j=1}^t (\deg_{l_j,m_j}K) ! &\mbox{otherwise }
\end{cases}$$
Therefore 
$$
\dfrac{\partial^{\lambda_1+\dots+\lambda_t} \phi(0)}
{\partial^{\lambda_1}a_{l_1,m_1}\dots\partial^{\lambda_t}a_{l_t,m_t}}=
\left( \lambda_1!\right)\cdots\left( \lambda_t!\right)e_J,  
$$
where $J\in \Lambda$ is characterized by
$$
\deg_{l,m}J =
\begin{cases}
\lambda_j &\mbox{if } (l,m)=(l_j,m_j) \mbox{ for some } j.\\
0 &\mbox{otherwise }
\end{cases}
$$
Note that $d(J,I)\!=\!\lambda_1+\dots+\lambda_t$. Conversely every $J\in \Lambda$
with $d(J,I)\!=\!\lambda_1+\dots+\lambda_t$ can be obtained in this way. Therefore, for every $0\leq s \leq d$, we have
$$
\left\langle \dfrac{\partial^s \phi(0)}{\partial^{\lambda_1}a_{l_1,m_1}\dots\partial^{\lambda_t}a_{l_t,m_t}}
\: \big| \: 1\leq l_1,\dots,l_t\leq r, 1\leq m_j \leq n_j, j=1,\dots,t
\right\rangle=
\left\langle e_J | d(J,I)=s\right\rangle,
$$
and hence $T^s_{e_I}(SV_{\pmb{d}}^{\pmb{n}})= \left\langle e_J \: | \: d(I,J)\leq s\right\rangle$.
\end{proof}

\begin{Corollary}
For any point $p\in SV^{\pmb n}_{\pmb d}$ we have
$$
\dim T^s_p SV^{\pmb n}_{\pmb d}=\sum_{l=1}^s 
\sum_{\substack{0\leq l_1\leq d_1,\dots, 0\leq l_r \leq d_r \\ l_1+\dots +l_r=l}}
\!\!\!\!\binom{n_1+l_1-1}{l_1}\cdots\binom{n_r+l_r-1}{l_r}
$$
for any $0\leq s\leq d$, while $T^s_{p}(SV_{\pmb{d}}^{\pmb{n}})=\P^{N(\pmb{n},\pmb{d})}$ for any $s\geq d$.\\
In particular, 
$$
\dim T^s_p V^{n}_{d}=n+\binom{n+1}{2}+\dots+\binom{n+s-1}{s}
$$
for any $0\leq s\leq d$.
\end{Corollary}

%
%

\section{Osculating projections}\label{oscproj}
In this section we study linear projections of Segre-Veronese varieties from their osculating spaces.
We follow the notation introduced in the previous section.

We start by analyzing projections of Veronese varieties from osculating spaces at coordinate points.
We consider a Veronese variety $V_{d}^{n}\subset \mathbb{P}^{N(n,d)}$, $d\geq 2$, and a coordinate point $e_{\underline{i}}=e_{(i,\dots,i)}\in V_{d}^{n}$
for some $i\in\{0,1,\dots,n\}$.
We write $(z_I)_{I\in \Lambda_{n,d}}$ for the coordinates in $\mathbb{P}^{N(n,d)}$. 
The linear projection
\begin{align*}
\pi_{i}:\P^{n}&\dasharrow \P^{n-1}\\
(x_j)
&\mapsto (x_j)_{j\neq i}
\end{align*}
induces the linear projection
\begin{align}\label{projonepoint}
\Pi_{i}:V_{d}^{n}&\dasharrow V_{d}^{n-1}\\
(z_I)_{I\in \Lambda_{n,d}}
&\mapsto (z_I)_{I\in \Lambda_{n,d} \: | \: i\notin I}\nonumber
\end{align}
making the following diagram commute
\[
\begin{tikzpicture}[xscale=6.5,yscale=-1.7]
    \node (A0_0) at (0, 0) {$\mathbb{P}^{n}$};
    \node (A0_1) at (1, 0) {$V_{d}^{n}\subseteq \mathbb{P}^{N(n,d)}$};
    \node (A1_0) at (0, 1) {$\mathbb{P}^{n-1}$};
    \node (A1_1) at (1, 1) {$V_{d}^{n-1}\subseteq \mathbb{P}^{N(n-1,d)}$};
    \path (A0_0) edge [->]node [auto] {$\scriptstyle{\nu_{d}^{n}}$} (A0_1);
    \path (A1_0) edge [->]node [auto] {$\scriptstyle{\nu_{d}^{n-1}}$} (A1_1);
    \path (A0_1) edge [->,dashed]node [auto] {$\scriptstyle{\Pi_{i}}$} (A1_1);
    \path (A0_0) edge [->,dashed]node [auto,swap] {$\scriptstyle{\pi_{i}}$} (A1_0);
  \end{tikzpicture} 
\]

\begin{Lemma}\label{projoscveronese}
Consider the projection of the Veronese variety $V_{d}^{n}\subset \mathbb{P}^{N(n,d)}$, $d\geq 2$, from 
the osculating space ${T_{e_{\underline{i}}}^s}$ of order $s$ at the point $e_{\underline{i}}=e_{(i,\dots,i)}\in V_{d}^{n}$, $0\leq s\leq d-1$:
$$
\Gamma_i^s:V_{d}^{n}\dasharrow \P^{N(n,d,s)}.
$$
Then $\Gamma_i^s$ is birational for any $s\leq d-2$, while $\Gamma_i^{d-1}=\Pi_{i}$.
\end{Lemma}

\begin{proof}
The case $s=d-1$ follows from Proposition~\ref{oscsegver} and the expression in \eqref{projonepoint} above, observing that, for any $J\in\Lambda_{n,d}$, 
$$
d(J,(i,\dots,i))=d \Leftrightarrow i\notin J.
$$ 

Since $\Gamma_i^{d-2}$ factors through $\Gamma_i^{j}$ for every $0\leq j\leq d-3$, it is enough to prove  birationality of $\Gamma_i^{d-2}$.
We may assume that $i\neq 0$, and consider the collection of indices
$$
J_0=(0,\dots,0,0),J_1=(0,\dots,0,1),\dots,J_{n}=(0,\dots,0,n)\in \Lambda_{(n,d)}.
$$
Note that $d(J_j,(i,\dots,i))\geq d-1$ for any $j\in\{1, \dots, n\}$. 
So we can define the linear projection
\begin{align*}
\gamma:\P^{N(n,d,s)} & \dasharrow \P^{n} .\\
(z_J)_{J \: |\: d(I,J)>d-2} &\mapsto (z_{J_0},\dots,z_{J_{n}})
\end{align*}
The composition 
\begin{align*}
\gamma\circ \Gamma_i^{d-2}\circ \nu_{d}^{n}:\mathbb{P}^n&\dasharrow\mathbb{P}^n\\
(x_0:\cdots:x_{n})&\mapsto (x_0^{d-1}x_0:\cdots x_0^{d-1}x_{n})=(x_0:\cdots:x_{n}) 
\end{align*}
is the identity, and thus $\Gamma_i^{d-2}$ is  birational.
\end{proof}

Now we turn to Segre-Veronese varieties. 
Let $SV_{\pmb{d}}^{\pmb{n}}\subset \P^{N(\pmb{n},\pmb{d})}$ be a Segre-Veronese variety, and consider a coordinate point
$$
e_I = e_{i_1}^{d_1}\otimes e_{i_2}^{d_2}\otimes\dots\otimes e_{i_r}^{d_r}\in SV_{\pmb{d}}^{\pmb{n}},
$$
with $0\leq i_j\leq n_j$, $I=\big((i_1,\dots,i_1),\dots, (i_r,\dots,i_r)\big)$. 
We write $(z_I)_{I\in \Lambda}$ for the coordinates in $\mathbb{P}^{N(\pmb{n},\pmb{d})}$.
Recall from Proposition~\ref{oscsegver} that the linear projection of $SV_{\pmb{d}}^{\pmb{n}}$ 
from the osculating space $T_{e_I}^s$ of order $s$ at $e_I$ is given by 
\begin{align}\label{eq:osc_proj}
\Pi_{T_{e_I}^s}:SV_{\pmb{d}}^{\pmb{n}}&\dasharrow \P^{N(\pmb{n},\pmb{d},s)}\\
(z_J)
&\mapsto (z_J)_{J\in\Lambda \: | \: d(I,J)>s}\nonumber
\end{align}
for every $s\leq d -1$.


In order to study the fibers of $\Pi_{T_{e_I}^s}$, we define auxiliary rational maps 
$$
\Sigma_l:SV_{\pmb{d}}^{\pmb{n}}\dasharrow \P^{N_l}
$$
for each $l\in \{1,\dots, r\}$ as follows. The map $\Sigma_1$ is the composition of the product map 
$$
\Gamma_{i_1}^{d_1-2} \times \prod_{j=2}^r \Pi_{i_j}:
V_{d_1}^{n_1}\times \dots \times V_{d_r}^{n_r}
\dasharrow  \mathbb P^{N(n_1,d_1,d_1-2)} \times  \prod_{j=2}^r \P^{N(n_j-1,d_j)}
$$ 
with the Segre embedding 
$$
\mathbb P^{N(n_1,d_1,d_1-2)} \times  \prod_{j=2}^r \P^{N(n_j-1,d_j)}\hookrightarrow \P^{N_1}.
$$
The other maps $\Sigma_l$, $2\leq l\leq r$, are defined analogously. 
In coordinates we have:
\begin{align}\label{eq_sigmal}
\Sigma_l:SV_{\pmb{d}}^{\pmb{n}}&\dasharrow \P^{N_l} \ , \\
(z_J)
&\mapsto (z_J)_{J\in\Lambda_l}\nonumber
\end{align}
where $\Lambda_l=\big\{J=(J^1,\dots,J^r)\in \Lambda \ \big| \ d(J^l,(i_l,\dots,i_l))\geq d_l-1 \text{ and }
i_j\not\in J^j \text{ for } j\neq l\}.$


\begin{Proposition}\label{projosconept}
Consider the projection of the Segre-Veronese variety $SV_{\pmb{d}}^{\pmb{n}}\subset \P^{N(\pmb{n},\pmb{d})}$ from 
the osculating space $\Pi_{T_{e_I}^s}$ of order $s$ at the point 
$e_I = e_{i_1}^{d_1}\otimes e_{i_2}^{d_2}\otimes\dots\otimes e_{i_r}^{d_r}\in SV_{\pmb{d}}^{\pmb{n}}$, $0\leq s\leq d-1$:
$$
\Pi_{T_{e_I}^s}:SV_{\pmb{d}}^{\pmb{n}} \dasharrow \P^{N(\pmb{n},\pmb{d},s)}.
$$
Then $\Pi_{T_{e_I}^s}$ is birational for any $s\leq d-2$, while $\Pi_{T_{e_I}^{d-1}}$ fits in the following commutative diagram:
\[
\begin{tikzpicture}[xscale=6.5,yscale=-1.7]
    \node (A0_0) at (0, 0) {$\mathbb{P}^{n_1}\times\dots\times\mathbb{P}^{n_r}$};
    \node (A0_1) at (1, 0) {$SV_{\pmb{d}}^{\pmb{n}}\subseteq \mathbb{P}^{N(\pmb{d},\pmb{n})}$};
    \node (A1_0) at (0, 1) {$\mathbb{P}^{n_1-1}\times\dots\times\mathbb{P}^{n_r-1}$};
    \node (A1_1) at (1, 1) {$SV_{\pmb{d}}^{\pmb{n}-\pmb{1}}\subseteq \mathbb{P}^{N(\pmb{d},\pmb{n}-\pmb{1})}$};
    \path (A0_0) edge [->]node [auto] {$\scriptstyle{\sigma\nu_{\pmb{d}}^{\pmb{n}}}$} (A0_1);
    \path (A1_0) edge [->]node [auto] {$\scriptstyle{\sigma\nu_{\pmb{d}}^{\pmb{n}-\pmb{1}}}$} (A1_1);
    \path (A0_1) edge [->,dashed]node [auto] {$\scriptstyle{\Pi_{T_{e_I}^{d -1}}}$} (A1_1);
    \path (A0_0) edge [->,dashed]node [auto,swap] {$\scriptstyle{\pi_{i_1}\times\dots\times \pi_{i_r}}$} (A1_0);
  \end{tikzpicture} 
\]
where $\pmb{n}-\pmb{1}=(n_1-1,\dots,n_r-1).$
Furthermore, the closure of the fiber of $\Pi_{T_{e_I}^{d -1}}$ is the Segre-Veronese variety 
$SV_{\pmb{d}}^{1,\dots,1} $.
\end{Proposition}

\begin{proof}
The case $s=d-1$ follows from the expressions in \eqref{projonepoint} and \eqref{eq:osc_proj}, and Lemma~\ref{projoscveronese}.

Since $\Pi_{T_{e_I}^{d-2}}$ factors through $\Pi_{T_{e_I}^{j}}$ for every $0\leq j\leq d-3$, it is enough to prove  birationality of $\Pi_{T_{e_I}^{d-2}}$.

First note that $\Pi_{T_{e_I}^{d-2}}$ factors the map $\Sigma_l$ for any $l=1,\dots,r$. 
This follows from the expressions in \eqref{eq:osc_proj} and \eqref{eq_sigmal}, observing that 
$$
J=(J^1,\dots,J^r) \in \Lambda_l\Rightarrow 
d(J,I)\geq d_l-1+\sum_{j\neq l} d_j =d-1>d-2.
$$
We write $\tau_l:\mathbb P^{N(\pmb{n},\pmb{d},d-2)}\dasharrow \mathbb{P}^{N_l}$ for the projection making the 
following diagram  commute: 
\[
  \begin{tikzpicture}[xscale=3.5,yscale=-1.5]
    \node (A0_0) at (0, 0) {$SV_{\pmb{d}}^{\pmb{n}}$};
    \node (A1_0) at (1, 1) {$\mathbb{P}^{N_l}$};
    \node (A1_1) at (1, 0) {$\mathbb P^{N(\pmb{n},\pmb{d},d-2)}$};
    \path (A0_0) edge [->,swap, dashed] node [auto] {$\scriptstyle{\Sigma_l}$} (A1_0);
    \path (A1_1) edge [->, dashed] node [auto] {$\scriptstyle{\tau_l}$} (A1_0);
    \path (A0_0) edge [->, dashed] node [auto] {$\scriptstyle{\Pi_{T_{e_I}^{d-2}}}$} (A1_1);
  \end{tikzpicture}
\]

Take a general point
$$
x\in \overline{\Pi_{T_{e_I}^{d-2}}\left( SV_{\pmb{d}}^{\pmb{n}} \right)}\subseteq\mathbb P^{N(\pmb{n},\pmb{d},d-2)},
$$ 
and set $x_l=\tau_l(x)$, $l=1,\dots,r$. 
Denote by $F\subset \mathbb P^{\pmb{n}} $ the closure of the  fiber of $\Pi_{T_{e_I}^{d-2}}$
over $x$, and by $F_l$ the closure of the  fiber of $\Sigma_l$ over $x_l$. 
Let $y\in F\subset F_l$ be a general point, and write
$y=\sigma\nu^{\pmb{n}}_{\pmb{d}}(y_1,\dots,y_r)$, with $y_j\in \mathbb{P}^{n_j}, j=1,\dots,r$. 
By Lemma \ref{projoscveronese}, $F_l$ is the image under $\sigma\nu^{\pmb{n}}_{\pmb{d}}$ of 
$$
\langle y_1,e_{i_1}\rangle \times \dots \times \langle y_{l-1},e_{i_{l-1}}\rangle \times y_l \times
\langle y_{l+1},e_{i_{l+1}}\rangle \times \dots \times\langle y_r,e_{i_r}\rangle
\subset \mathbb{P}^{\pmb{n}},
$$
$l=1,\dots,r$.  It follows that $F=\{y\}$, and so
$\Pi_{T_{e_I}^{d-2}}$ is birational.
\end{proof}

Next study linear projections from the span of several osculating spaces at coordinate points, and investigate 
when they are birational.

\medskip

We start with the case of a Veronese variety $V_{d}^{n}\subset \mathbb{P}^{N(n,d)}$, with coordinate points 
$e_{\underline{i}}=e_{(i,\dots,i)}\in V_{d}^{n}$, $i\in\{0,1,\dots,n\}$.
For $m\in \{1, \dots, n\}$, let $\pmb{s}=(s_0,\dots,s_m)$ be an $(m+1)$-uple of positive integers, and set $s=s_0+\dots+s_m$.
Let $e_{\underline{i}_0}, \dots, e_{\underline{i}_{m}}\in V_{d}^{n}$ be distinct coordinate points, and denote by 
$T_{e_{\underline{i}_{0}}, \dots, e_{\underline{i}_{m}}}^{s_0,\dots, s_m}\subset \mathbb{P}^{N(n,d)}$ the linear span 
$\left\langle T_{e_{\underline{i}_{0}}}^{s_0},\dots,T_{e_{\underline{i}_{m}}}^{s_m}\right\rangle$.
By Proposition~\ref{oscsegver}, the projection of $V_{d}^{n}$ from 
$T_{e_{\underline{i}_{0}}, \dots, e_{\underline{i}_{m}}}^{s_0,\dots, s_m}\subset \mathbb{P}^{N(n,d)}$ is given by:
\begin{align}
\Gamma^{s_0,\dots, s_m}_{e_{\underline{i}_{0}}, \dots, e_{\underline{i}_{m}}}:V_{d}^{n}&\dasharrow \P^{N(n,d,\pmb{s})},\\
(z_I)_{I\in \Lambda_{n,d}}
&\mapsto (z_J)_{J\in \Lambda_{n,d}^{\pmb{s}}}\nonumber
\end{align}
whenever $\Lambda_{n,d}^{\pmb{s}}=\{J\in \Lambda_{n,d} \ |\  d\big(J,(j,\dots,j)\big)>s_j \text{ for } j=0,\dots,m\}$ is not empty.

\begin{Lemma}\label{projoscveroneseseveralpoints}
Let the notation be as above, and assume that  $d\geq 2$ and $0\leq s_j\leq d-2$ for $j=0,\dots,m$.
\begin{itemize}
\item[(\textit{a})] If $n\leq d$ and $s \leq n(d-1)-2$, then $\Gamma^{s_0,\dots, s_m}_{e_{\underline{i}_{0}}, \dots, e_{\underline{i}_{m}}}$ is birational onto its image.
\item[(\textit{b})] If $n\leq d$ and $s = n(d-1)-1$, then $\Gamma^{s_0,\dots, s_m}_{e_{\underline{i}_{0}}, \dots, e_{\underline{i}_{m}}}$ is a constant map.
\item[(\textit{c})] If $n> d$, then  $\Gamma^{d-2,\dots, d-2}_{e_{\underline{i}_{0}}, \dots, e_{\underline{i}_{n}}}$ is birational onto its image.
\end{itemize}
\end{Lemma}

\begin{proof}
Assume that $n\leq d$ and $s \leq n(d-1)-2$.
In order to prove that $\Gamma^{s_0,\dots, s_m}_{e_{\underline{i}_{0}}, \dots, e_{\underline{i}_{m}}}$ is birational, 
we will exhibit $J_0,\dots,J_{n}\in \Lambda_{n,d}^{\pmb{s}}$ , and linear projection 
\begin{align}\label{proj_cremona}
\gamma: \mathbb{P}^{N(n,d)}&\dasharrow \P^{n}\\
(z_I)_{I\in \Lambda_{n,d}^{\pmb{s}}}
&\mapsto (z_{J_j})_{j=0,\dots,n}\nonumber
\end{align}
such that the composition $\gamma\circ\Gamma^{s_0,\dots, s_m}_{e_{\underline{i}_{0}}, \dots, e_{\underline{i}_{m}}}\circ \nu_{d}^{n}:\P^{n}\dasharrow \P^{n}$
is the standard Cremona transformation of $\mathbb{P}^n$.
The $d$-uples $J_j\in \Lambda_{n,d}^{\pmb{s}}$ are constructed as follows. 
Since $n\leq d$ we can take $n$ of the coordinates of $J_j$ to be $0,1,\dots, \widehat{j},\dots, n$. 
The condition  $s \leq n(d-1)-2$ assures that we can complete the $J_j$'s by choosing $d-n$ common coordinates 
in such a way that, for every $i,j\in \{0,\dots, n\}$, we have $d\big(J_j,(i,\dots,i)\big)>s_i$
(i.e., $J_j$ has at most $(d-s_i-1)$ coordinates equal to to $i$). 
This gives $J_j\in \Lambda_{n,d}^{\pmb{s}}$ for every $j\in \{0,\dots, n\}$.
For the linear projection \eqref{proj_cremona} given by these $J_j$'s, we have that 
$\gamma\circ\Gamma^{s_0,\dots, s_m}_{e_{\underline{i}_{0}}, \dots, e_{\underline{i}_{m}}}\circ \nu_{d}^{n}:\P^{n}\dasharrow \P^{n}$
is the standard Cremona transformation of $\mathbb{P}^n$.

Now assume that $n\leq d$ and $s = n(d-1)-1$. 
If $J\in \Lambda_{n,d}^{\pmb{s}}$,
then $J$ has at most $d-s_i-1$ coordinates equal to $i$ for any $i\in \{0,\dots,n\}$. Since 
$$
\sum_{j=0}^{n} \left( d-s_j-1\right)=(n+1)(d-1)-s=d,
$$
there is only one possibility for $J$, i.e., $\Lambda_{n,d}^{\pmb{s}}$ has only one element, and so
$\Gamma^{s_0,\dots, s_m}_{e_{\underline{i}_{0}}, \dots, e_{\underline{i}_{m}}}$ is a constant map.

Finally, assume that $n>d.$ 
Set $K_0=\{0,\dots,{n-d}\}$. For any $j\in K_0$, set
$$
(J_{K_0})_j:=(j, {n-d+1},\dots, {n}),
$$
and note that $d\big((J_{K_0})_j,(i,\dots,i)\big)>d-2$ for every $i\in \{0,\dots, n\}$.
Thus $(J_{K_0})_j\in \Lambda_{n,d}^{\pmb{d-2}}$ for every $j\in {K_0}$. 
So we can define the linear projection
\begin{align*}
\gamma_{K_0}: \P^{N(n,d,\pmb{d-2})}&\dasharrow \P^{n-d}.\\
(z_I)_{I\in \Lambda_{n,d}^{\pmb{d-2}}}
&\mapsto (z_{(J_{K_0})_j})_{j\in {K_0}}
\end{align*}
The composition $\gamma_{K_0}\circ\Gamma^{d-2,\dots, d-2}_{e_{\underline{i}_{0}}, \dots, e_{\underline{i}_{n}}}\circ \nu_{d}^{n}:\P^{n}\dasharrow \P^{n-d}$
is the linear projection given by 
\begin{align*}
\gamma_{K_0}\circ\Gamma^{d-2,\dots, d-2}_{e_{\underline{i}_{0}}, \dots, e_{\underline{i}_{n}}}\circ \nu_{d}^{n}:\P^{n}&\dasharrow \P^{n-d}.\\
(x_i)_{i\in\{0,\dots,n\}}&\mapsto (x_i)_{i\in {K_0}}
\end{align*}
Analogously, for each subset $K\subset \{0,\dots,n\}$ with $n-d+1$ distinct elements,
we define a linear projection $\gamma_{K}: \P^{N(n,d,\pmb{d-2})}\dasharrow \P^{n-d}$ such that the composition 
$\gamma_{K}\circ\Gamma^{d-2,\dots, d-2}_{e_{\underline{i}_{0}}, \dots, e_{\underline{i}_{n}}}\circ \nu_{d}^{n}:\P^{n}\dasharrow \P^{n-d}$
is the linear projection given by
\begin{align*}
\gamma_{K}\circ\Gamma^{d-2,\dots, d-2}_{e_{\underline{i}_{0}}, \dots, e_{\underline{i}_{n}}}\circ \nu_{d}^{n}:\P^{n}&\dasharrow \P^{n-d}.\\
(x_i)_{i\in\{0,\dots,n\}}&\mapsto (x_i)_{i\in {K}}
\end{align*}
This shows that $\Gamma^{d-2,\dots, d-2}_{e_{\underline{i}_{0}}, \dots, e_{\underline{i}_{n}}}$ is birational. 
\end{proof}

The following is an immediate consequence of Lemma \ref{projoscveroneseseveralpoints}.

\begin{Corollary}\label{projoscveroneseseveralpointscor}
Let the notation be as above, and assume that  $d\geq 2$. Then
\begin{itemize}
\item[(\textit{a})] $\Gamma^{d-2,\dots,d-2}_{e_{\underline{i}_{0}}, \dots, e_{\underline{i}_{n-1}}}$ is birational.
\item[(\textit{b})] If $n\geq 2$ then
$\Gamma^{d-2,\dots,d-2,\min\{n,d\}-2}_{e_{\underline{i}_{0}}, \dots, e_{\underline{i}_{n}}}$ is birational, while
$\Gamma^{d-2,\dots,d-2,\min\{n,d\}-1}_{e_{\underline{i}_{0}}, \dots, e_{\underline{i}_{n}}}$ is not.
\item[(\textit{c})] If $d\geq 3$ then
$\Gamma^{d-3,\dots,d-3,\min\{2n,d\}-2}_{e_{\underline{i}_{0}}, \dots, e_{\underline{i}_{n}}}$ is birational, while
$\Gamma^{d-3,\dots,d-3,\min\{2n,d\}-1}_{e_{\underline{i}_{0}}, \dots, e_{\underline{i}_{n}}}$ is not.
\end{itemize}
\end{Corollary}

Now we turn to Segre-Veronese varieties. 
Let $SV_{\pmb{d}}^{\pmb{n}}\subset \P^{N(\pmb{n},\pmb{d})}$ be a Segre-Veronese variety, 
and write $(z_I)_{I\in \Lambda_{n,d}}$ for coordinates in $\mathbb{P}^{N(n,d)}$. 
Consider the coordinate points $e_{I_0},e_{I_1},\dots,e_{I_{n_1}}\in SV_{\pmb{d}}^{\pmb{n}}$, where
$$
I_j=((j,\dots,j),\dots,(j,\dots,j))\in \Lambda.
$$
(Recall that $n_1\leq\dots \leq n_r$.)
Let $\pmb{s}=(s_0,\dots,s_m)$ be an $(m+1)$-uple of positive integers, and set $s=s_0+\dots+s_m$.
Denote by $T_{e_{I_0},\dots,e_{I_m}}^{s_0,\dots,s_m}\subset \P^{N(\pmb{n},\pmb{d})}$ the linear span of the osculating spaces 
$T_{e_{I_0}}^{s_0},\dots,T_{e_{I_m}}^{s_m}$.
By Proposition~\ref{oscsegver}, the projection of $SV_{\pmb{d}}^{\pmb{n}}$ from 
$T_{e_{I_0},\dots,e_{I_m}}^{s_0,\dots,s_m}$ is given by:
\begin{align}\label{eq:multi_osc_proj}
\Pi_{T_{e_{I_0},\dots,e_{I_m}}^{s_0,\dots,s_m}}:
SV_{\pmb{d}}^{\pmb{n}}&\dasharrow \P^{N(\pmb{d},\pmb{n},\pmb{s})}\\
(z_J)_{J \in \Lambda}
&\mapsto (z_J)_{J \in \Lambda^{\pmb{s}}} \nonumber
\end{align}
whenever $\Lambda^{\pmb{s}}\!=\!\{J \in \Lambda\: |\: d(I_j,J)>s_j \ \forall j\}$ is not empty. 

\begin{Proposition}\label{projoscseveralpoints} 
Let the notation be as above, and assume that $r,d\geq 2$. 
Then the projection  
$\Pi_{T_{e_{I_0},\dots,e_{I_{n_1-1}}}^{d-2,\dots,d-2}}:SV_{\pmb{d}}^{\pmb{n}} \dasharrow \P^{N(\pmb{d},\pmb{n},{\pmb{d-2}})}$ 
is birational.
\end{Proposition}

\begin{proof}
For each $l\in\{1,\dots, r\}$, set 
$$
\Lambda_l=\left\{J\!=\!(J^1,\dots,J^r)\in \Lambda \: \big| \:
\begin{cases} 
0,\dots, n_1-1\not \in J^j \text{ if } j\neq l \\
d\big(J^l,(i,\dots, i)\big)\geq d_l-1 \ \forall i\in \{0,\dots, n_1-1\}
\end{cases}
\right\},
$$
and consider the linear projection 
\begin{align}\label{Sigma_l}
\Sigma_l:SV_{\pmb{d}}^{\pmb{n}}
&\dasharrow \mathbb P^{N_l}.\\
{\left(z_J\right)}_{J\in\Lambda}&\mapsto
{\left( z_J\right)}_{J\in\Lambda_l}\nonumber
\end{align}
Note that $\Lambda_l\subset \Lambda^{\pmb{d-2}}$, and so there is a linear projection
$\tau_l:\P^{N(\pmb{d},\pmb{n},\pmb{d-2})}\dasharrow \mathbb P^{N_l}$
such that $\Sigma_l=\tau_l\circ \Pi_{T^{d-2,\dots,d-2}_{p_0,\dots,p_{n_1-1}}}$.

The restriction of $\Sigma_l\circ \sigma\nu^{\pmb{n}}_{\pmb{d}}$ to 
$$
\{pt\} \times \dots \{pt\} \times \P^{n_l} \times \{pt\} \times \dots \{pt\}
$$
is isomorphic to the osculating projection 
$$
\Gamma^{d_l-2,\dots,d_l-2}_{e_{\underline{i}_{0}}, \dots, e_{\underline{i}_{n_1-1}}}:V_{d_l}^{n_l}\dasharrow \P^{N(n_l,d_l,\pmb{d_l-2})}.
$$
This is birational by Corollary~\ref{projoscveroneseseveralpointscor}. 
For $j\neq l$, the restriction of $\Sigma_l\circ \sigma\nu^{\pmb{n}}_{\pmb{d}}$ to 
$$
\{pt\} \times \dots \{pt\} \times \P^{n_j} \times \{pt\} \times \dots \{pt\}
$$
is isomorphic to the projection with center $\left\langle e_0,\dots,e_{n_1-1} \right\rangle$.
Arguing as in the last part of the proof of Proposition \ref{projosconept}, we conclude that 
$\Pi_{T_{e_{I_0},\dots,e_{I_{n_1-1}}}^{d-2,\dots,d-2}}$ is birational.
\end{proof}

%
%

\section{Non-secant defectivity of Segre-Veronese varieties}\label{mainsec}

In this section we explain how osculating projections can be used to establish non-secant defectivity of Segre-Veronese varieties.
We start by recalling the definition of secant varieties and secant defectivity.

\begin{Definition}[Secant varieties]
Let $X\subset\P^N$ be a non-degenerate projective variety of dimension $n$. 
Consider the rational map $\alpha: X\times  \dots \times X \dasharrow \G(h-1,N)$
mapping $h$ general points to their linear span $\langle x_1, \dots , x_{h}\rangle$. Let 
$$
\Gamma_h(X)\subset X\times \dots \times X\times\G(h-1,N)
$$
be the closure of the graph of $\alpha$, with the natural projection $\pi_2:\Gamma_h(X)\to\G(h-1,N)$.
Set 
$$
\mathcal{S}_h(X)=\pi_2(\Gamma_h(X))\subset\G(h-1,N).
$$
Both $\Gamma_h(X)$ and $\mathcal{S}_h(X)$ are irreducible of dimension $hn$.  
Now consider the incidence variety 
$$
\mathcal{I}_h=\{(x,\Lambda) \: | \: x\in \Lambda\} \subset\P^N\times\G(h-1,N),
$$
and the associated diagram
$$
\xymatrix{
&\ \ \ \ \ \ {\mathcal{I}_h} \ \ \ \  \ar[dl]_{\pi_h}\ar[dr]^{\psi_h}&\\
\P^N&&\G(h-1,N).
}
$$
Note that
$\psi_h:\mathcal{I}_h\to\G(h-1,N)$  is a $\P^{h-1}$-bundle.
The variety
$$
(\psi_h)^{-1}(\mathcal{S}_h(X))\subset \mathcal{I}_h \subset\P^N\times\G(h-1,N)
$$
is an $(hn+h-1)$-dimensional variety with a $\P^{h-1}$-bundle structure over $\mathcal{S}_h(X)$.

The {\it $h$-secant variety} of $X$ is the variety
$$
\sec_{h}(X)=\pi_h\big((\psi_h)^{-1}(\mathcal{S}_h(X))\big)\subset\P^N.
$$

We say that $X$ is \textit{$h$-defective} if 
$$
\dim\sec_{h}(X)<\min\{\dim\Sec_{h}(X),N\}.
$$
\end{Definition}

Determining secant defectivity is a classical problem  in algebraic geometry. 
The following characterization of secant defectivity in terms of tangential projections is due to Chiantini and Ciliberto.

\begin{Definition}
Let $x_1,\dots,x_h\in X\subset\mathbb{P}^N$ be general points, with tangent spaces $T_{x_i}X$. 
We say that  the linear projection
$$
\tau_{X,h}:X\subseteq\mathbb{P}^N\dasharrow\mathbb{P}^{N_h}
$$
with center $\left\langle T_{x_1}X,\dots,T_{x_h}X\right\rangle$ is a \textit{general $h$-tangential projection}. 
\end{Definition}

\begin{Proposition}\cite[Proposition 3.5]{CC01}\label{cc}
Let $X\subset\P^N$ be a non-degenerate projective variety of dimension $n$, and $x_1,\dots,x_h\in X$  general points. Assume that 
$$
N-\dim(\left\langle T_{x_1}X,\dots,T_{x_h}X\right\rangle)-1\geq n.
$$
Then the general $h$-tangential projection $\tau_{X,h}:X\dasharrow X_h$ is generically finite if and only if 
$X$ is not $(h+1)$-defective.
\end{Proposition}

In general, however, it is hard to control the dimension of the fibers of tangential projections $\tau_{X,h}$ when $h$ is large. 
In \cite{MRi16} a new strategy was introduced, based on degenerating the linear span of several tangent spaces 
$T_{x_i}X$ into a subspace contained in a
single osculating space $T_x^{k}X$. 
The more points one can use in this degeneration, the better the method works. 
To count the number of points that can be used, the following notion was introduced in \cite[Definition 4.6 and Assumption 4.3]{MRi16}.

\begin{Definition}\label{moscularity}
Let $X\subset\mathbb{P}^N$ be a projective variety.
 
We say that $X$ has \textit{$m$-osculating regularity} if the following property holds. 
Given general points $p_1,\dots,p_{m}\in X$  and  integer $k\geq 0$, 
there exists a smooth curve $C$ and morphisms $\gamma_j:C\to X$, $j=2,\dots,m$, 
such that  $\gamma_j(t_0)=p_1$, $\gamma_j(t_\infty)=p_j$, and the flat limit $T_0$ in $\G(dim(T_t),N)$ of the family of linear spaces 
$$
T_t=\left\langle T^{k}_{p_1},T^{k}_{\gamma_2(t)},\dots,T^{k}_{\gamma_{m}(t)}\right\rangle,\: t\in C\backslash \{t_0\}
$$
is contained in $T^{2k+1}_{p_1}$.

We say that $X$ has \textit{strong $2$-osculating regularity} if the following property holds. 
Given general points $p,q\in X$ and  integers $k_1,k_2\geq 0$,
there exists a smooth curve $\gamma:C\to X$ such that  $\gamma(t_0)=p$, $\gamma(t_\infty)=q$ 
and the flat limit $T_0$ in $\G(dim(T_t),N)$ of the family of linear spaces 
$$
T_t=\left\langle T^{k_1}_p,T^{k_2}_{\gamma(t)}\right\rangle,\: t\in C\backslash \{t_0\}
$$
is contained in $T^{k_1+k_2+1}_p$.
\end{Definition}

The method of  \cite{MRi16} goes as follows. 
If $X\subset\mathbb{P}^N$ has $m$-osculating regularity, one degenerates a general $m$-tangential projection 
into a linear projection with center contained in $T^{3}_p$X. 
Then one further degenerates a general osculating projection $T^{(3,\dots,3)}_{p_1, \dots, p_m}$
into a linear projection with center contained in $T^{7}_qX$.
By proceeding recursively, one degenerates a general $h$-tangential projection 
into a linear projection with center contained in a suitable 
linear span of osculating spaces, and then check whether this projection is generically finite (see the proof of \cite[Theorem 5.3]{MRi16} for details). So one gets the following criterion:
\begin{Theorem}
Let $X\subset \P^N$ be a projective variety having $m$-osculating regularity. Let $k_1,\dots,k_l\geq 1$ be integers such that the general osculating projection 
$\Pi_{T^{k_1,\dots,k_l}_{p_1,\dots,p_l}}$ is generically finite.
Then $X$ is not $h$-defective for $h\leq \displaystyle \left(\sum_{j=1}^{l}m^{\lfloor \log_2(k_j+1)\rfloor -1}\right)+1$.
\end{Theorem}

If $X$ in addition has strong $2$-osculating regularity, then this can be done even more effectively. 
To state the criterion of \cite{MRi16}, we introduce a function $h_m:\N_{\geq 0}\to \N_{\geq 0}$ counting how many tangent spaces 
can be degenerated into a higher order osculating space in this way. 

\begin{Definition}\label{defhowmanytangent}
Let $m\geq 2$ be an integer. Define the function
\begin{align*}
h_m:\N_{\geq 0}\to \N_{\geq 0}
\end{align*}
as follows: $h_m(0)=0$. For any $k\geq 1$, write
$$
k+1=2^{\lambda_1}+2^{\lambda_2}+\dots+2^{\lambda_l}+\varepsilon,
$$
where $\lambda_1>\lambda_2>\dots>\lambda_l \geq 1$, $\varepsilon\in \{0,1\}$. Then set
$$
h_m(k)=m^{\lambda_1-1}+m^{\lambda_2-1}+\dots+m^{\lambda_l-1}.
$$
\end{Definition}

\begin{Theorem}{\cite[Theorem 5.3]{MRi16}}\label{lemmadefectsviaosculating}
Let $X\subset \P^N$ be a projective variety 
having $m$-osculating regularity and strong $2$-osculating regularity.
Let $k_1,\dots,k_l\geq 1$ be integers such that the general osculating projection 
$\Pi_{T^{k_1,\dots,k_l}_{p_1,\dots,p_l}}$ is generically finite.
Then $X$ is not $h$-defective for $h\leq \displaystyle\sum_{j=1}^{l}h_m(k_j)+1$.
\end{Theorem}

We now prove our main result on non-defectivity of Segre-Veronese Varieties.
We follow the notation introduced in the previous sections. 

\begin{Theorem}\label{main}
Let the notation be as above. 
The Segre-Veronese variety $SV^{\pmb n}_{\pmb d}$ is not $h$-defective for
$$h\leq n_1h_{n_1+1}(d-2)+1.$$
\end{Theorem}

\begin{proof}
We will show in Propositions~\ref{limitosculatingspacessegver} and \ref{limitosculatingspacessegverII} that 
the Segre-Veronese variety $SV_{\pmb d}^{\pmb n}$ has 
strong $2$-osculating regularity, and $(n_1+1)$-osculating regularity.
The result then follows immediately from Proposition \ref{projoscseveralpoints} and Theorem~\ref{lemmadefectsviaosculating}.
\end{proof}

\begin{Remark}
Write 
$$
d-1 = 2^{\lambda_1}+2^{\lambda_2}+\dots + 2^{\lambda_s}+\epsilon
$$
with $\lambda_1 >\lambda_2>\dots >\lambda_s\geq 1$, $\epsilon\in\{0,1\}$, so that $\lambda_1 = \lfloor \log_2(d-1)\rfloor$.
By Theorem \ref{main} $SV^{\pmb n}_{\pmb d}$ is not $h$-defective for
$$
h\leq n_1((n_1+1)^{\lambda_1-1}+\dots + (n_1+1)^{\lambda_s-1})+1.
$$
So we have that asymptotically  $SV^{\pmb n}_{\pmb d}$ is not $h$-defective for
$$
h\leq n_1^{\lfloor \log_2(d-1)\rfloor}.
$$
\end{Remark}


Recall  \cite[Proposition 3.2]{CGG03}:
except for the Segre product $\mathbb{P}^{1}\times\mathbb{P}^{1}\subset\mathbb{P}^3$, 
the Segre-Veronese variety $SV^{\pmb n}_{\pmb d}$ is not $h$-defective for $h\leq \min\{n_i\}+1$,
independently of $\pmb d$.
In the following table, for a few values of $d$, we compute the highest value of $h$ for which Theorem~\ref{main} 
gives non $h$-defectivity of $SV^{\pmb n}_{\pmb d}$.

\medskip

\begin{center}
\begin{tabular}{|c|l|}
\hline 
$d = d_1+\dots +d_r$ & $h$\\ 
\hline 
$3$ & $n_1+1$\\ 
\hline 
$5$ & $n_1(n_1+1)+1$\\ 
\hline 
$7$ & $n_1((n_1+1)+1)+1$\\ 
\hline 
$9$ & $n_1(n_1+1)^2+1$\\ 
\hline 
$11$ & $n_1((n_1+1)^2+1)+1$\\ 
\hline 
$13$ & $n_1((n_1+1)^2+n_1+1)+1$\\ 
\hline
$15$ & $n_1((n_1+1)^2+(n_1+1)+1)+1$\\
\hline 
$17$ & $n_1(n_1+1)^3+1$\\
\hline
\end{tabular} 
\end{center}

\begin{Remark}
Note that the bound of Theorem \ref{main} is sharp in some cases. For instance, it is well known that $SV^{(1,1)}_{(2,2)}$, $SV^{(1,1,1)}_{(1,1,2)}$, $SV^{(1,1,1,1)}_{(1,1,1,1)}$ are $3$-defective, and $SV^{(2,2,2)}_{(1,1,1)}$ is $4$-defective. On the other hand $SV^{(1,1)}_{(2,2)}$, $SV^{(1,1,1)}_{(1,1,2)}$, $SV^{(1,1,1,1)}_{(1,1,1,1)}$ are not $2$-defective, and $SV^{(2,2,2)}_{(1,1,1)}$ is not $3$-defective.
\end{Remark}


\medskip

\begin{Remark}
By Proposition~\ref{limitosculatingspacessegverII},
the Segre-Veronese variety $SV_{\pmb d}^{\pmb n}$,  $\pmb{n}=(n_1\leq \dots \leq n_r)$, has $(n_1+1)$-osculating regularity. 
We do not know in general what is the highest osculating regularity of $SV_{\pmb d}^{\pmb n}$. 
Better osculating regularity results would yield better bounds for $h$ in Theorem~\ref{main}. 

Determining the highest osculating regularity of a variety can be a difficult problem. 
There are examples of ruled surfaces that do not  have $2$-osculating regularity (see  \cite[Example 4.4]{MRi16}). In the smooth case, we give below an example of a surface having $2$-osculating regularity but not $3$-osculating regularity.
\end{Remark}

\begin{Example}
Consider the rational normal scroll $X_{(1,7)}\subset\mathbb{P}^9$, which is parametrized by 
$$
\begin{array}{cccc}
\phi: & \mathbb{A}^2 & \longrightarrow & \mathbb{A}^9\\ 
 & (u,\alpha) & \mapsto & (\alpha u^7,\alpha u^6,\dots, \alpha u, \alpha, u)
\end{array} 
$$ 
Note that $\frac{\partial^2\phi}{\partial\alpha^2} = \frac{\partial^3\phi}{\partial\alpha^3} = \frac{\partial^3\phi}{\partial\alpha^2u} = 0$, while there are no other relations between the partial derivatives, up to order three, of $\phi$ at the general point of $X_{(1,7)}$. Therefore
$$\dim (T^3X_{(1,7)}) = 10-3-1 = 6$$
On the other hand by \cite[Lemma 4.10]{DP96} we have that $\dim(\sec_3(X_{(1,7)})) = 7$. Hence by Terracini's lemma \cite[Theorem 1.3.1]{Ru03} the span of three general tangent spaces of $X_{(1,7)}$ has dimension seven.
\end{Example}



%
%

\section{Osculating regularity of Segre-Veronese varieties}\label{degtanosc}

In this section we show that the Segre-Veronese variety $SV_{\pmb d}^{\pmb n}\subseteq\mathbb{P}^{N(\pmb{n},\pmb{d})}$ has 
strong $2$-osculating regularity, and $(n_1+1)$-osculating regularity.
We follow the notation introduced in the previous sections.

\begin{Proposition}\label{limitosculatingspacessegver}
The Segre-Veronese variety $SV_{\pmb{d}}^{\pmb{n}}\subseteq\mathbb{P}^{N(\pmb{n},\pmb{d})}$ has 
strong $2$-osculating regularity.
\end{Proposition}

\begin{proof}
Let $p,q\in SV_{\pmb{d}}^{\pmb{n}}\subseteq\mathbb{P}^{N(\pmb{n},\pmb{d})}$ be general points.
There is a projective automorphism of $SV_{\pmb{d}}^{\pmb{n}}\subseteq\mathbb{P}^{N(\pmb{n},\pmb{d})}$ mapping 
$p$ and $q$ to the coordinate points $e_{I_0}$ and $e_{I_1}$. 
These points are connected by the degree $d$ rational normal curve defined by 
$$
\gamma([t:s])=(se_{I_0}+te_{I_1})^{d_1}\otimes\dots\otimes (se_{I_0}+te_{I_1})^{d_r}.
$$
We work in the affine chart $(s=1)$, and set $t=(t:1)$. 
Given integers $k_1,k_2\geq 0$, consider the family of linear spaces 
$$
T_t=\left\langle T^{k_1}_{e_{I_0}},T^{k_2}_{\gamma(t)}\right\rangle,\: t\in \C\backslash \{0\}.
$$
We will show that the flat limit $T_0$  of $\{T_t\}_{t\in \C\backslash \{0\}}$ in
$\G(\dim(T_t),N(\pmb{n},\pmb{d}))$ is contained in $T^{k_1+k_2+1}_{e_{I_0}}$.

We start by writing the linear spaces $T_t$ explicitly. 
For $j=1,\dots, r$, we define the vectors 
$$
e_0^{t}=e_0+te_1,e_1^t=e_1,e_2^t=e_2,\dots,e_{n_j}^{t}=e_{n_j}\in V_{j}.
$$
Given $I^j=(i_1,\dots,i_{d_j})\in \Lambda_{n_j,d_j}$, we denote by 
$e_{I^j}^{t}\in \Sym^{d_j}V_j$  the symmetric product $e_{i_1}^{t}\cdot\ldots\cdot e_{i_{d_j}}^{t}$.
Given $I=(I^1,\dots,I^r)\in \Lambda=\Lambda_{\pmb{n},\pmb{d}}$, we denote by 
$e_I^{t}\in \mathbb{P}^{N(\pmb{n},\pmb{d})}$
the point corresponding to 
$$
e_{I^1}^{t}\otimes\dots\otimes e_{I^r}^{t}\in \Sym^{d_1}V_1\otimes\dots\otimes \Sym^{d_r}V_r.
$$
By Proposition \ref{oscsegver} we have 
$$
T_t=\left\langle e_I\: | \: d(I,I_0)\leq k_1; \: e_I^{t}\: |\: d(I,I_0)\leq k_2\right\rangle,\: t\neq 0.
$$
We shall write $T_t$ in terms of the basis $\{e_J|J\in\Lambda\}$. 
Before we do so, it is convenient to introduce some additional notation. 

\begin{Notation}\label{defdelta}
Let $I\in \Lambda_{n,d}$, and write:
\begin{equation}\label{defI}
I=(\underbrace{0,\dots,0}_{a\text{ times}},\underbrace{1,\dots,1}_{b \text{ times}},i_{a+b+1},\dots,i_d),
\end{equation}
with $a,b\geq 0$ and $1<i_{a+b+1}\leq\dots\leq i_d$.
Given $l\in \Z$, define $\delta^l(I)\in \Lambda_{n,d}$ as 
$$
\delta^l(I)=
(\underbrace{0,\dots,0}_{a-l\text{ times}},\underbrace{1,\dots,1}_{b+l \text{ times}},i_{a+b+1},\dots,i_d),
$$
provided that $-b\leq l\leq a$. 

Given $I=(I^1,\dots,I^r)\in \Lambda$ and $\pmb{l}=(l_1,\dots,l_r)\in \Z^r$, define
$$
\delta^{\pmb{l}}(I)=(\delta^{l_1}(I^1),\dots,\delta^{l_r}(I^r))\in \Lambda,
$$
provided that each $\delta^{l_j}(I_j)$ is defined. 
Let $l\in \Z$. If $l\geq 0$, set 
$$
\Delta(I,l)= \left\{ \delta^{\pmb{l}}(I) | \pmb{l}=(l_1,\dots,l_r), l_1,\dots,l_r\geq 0, \: l_1+\dots +l_r=l\:\right\}\subset \Lambda.
$$
If $l<0$, set 
$$
\Delta(I,l)=\left\{ J|\: I\in\Delta(J,-l) \right\} \subset \Lambda.
$$
Define also:
$$
s_I^+=\max_{l\geq 0}\{\Delta(I,l)\neq \emptyset\}\in\{0,\dots,d\}= d-d(I,I_0) ,
$$
$$
s_I^-=\max_{l\geq 0}\{\Delta(I,-l)\neq \emptyset\}\in\{0,\dots,d\}= d-d(I,I_1),
$$
$$
\Delta(I)^+=\bigcup_{0\leq l} \Delta(I,l)=\!\!\!\!\bigcup_{0\leq l \leq s_I^+} \!\!\!\! \Delta(I,l), \text{ and }
$$
$$
\Delta(I)^-=\bigcup_{0\leq l} \Delta(I,-l)=\!\!\!\!\bigcup_{0\leq l \leq s_I^-} \!\!\!\!\Delta(I,-l).
$$
Note that if $J\in\Delta(I,l)$, then $d(J,I)=|l|$, $d(J,I_0)=d(I,I_0)+l$, and $d(J,I_1)=d(I,I_1)-l$.
Note also that, if $J\in \Delta(I)^-\cap \Delta(K)^+$, then $d(I,K)=d(I,J)+d(J,K)$.
\end{Notation}

Now we write each vector $e_I^t$ with $d(I,I_0)< k_2$ in terms of the basis $\{e_J|J\in\Lambda\}$. 

First, we consider the Veronese case. Let $I=(i_1,\dots,i_d)\in \Lambda_{n,d}$ be as in \eqref{defI}, so that $s_I^+=a$.
We have:
\begin{align*}
e_I^t&=(e_0^t)^a (e_1^t)^b e_{i_{a+b+1}}^t\cdots e_{i_d}^t
=(e_0+te_1)^a e_1^b e_{i_{a+b+1}}\cdots e_{i_d}=\\
&=e_0^a e_1^b e_{i_{a+b+1}}\cdots e_{i_d}
+t\binom{a}{1}e_0^{a-1}e_1^{b+1} e_{i_{a+b+1}}\cdots e_{i_d}+\dots
+t^a  e_1^{b+a} e_{i_{a+b+1}}\cdots e_{i_d}=\\
&=\sum_{l=0}^a t^l\binom{a}{l}e_0^{a-l}e_1^{b+l} e_{i_{a+b+1}}\cdots e_{i_d}
=\sum_{l=0}^a t^l\binom{a}{l}e_{\delta^l (I)}.
\end{align*}

In the Segre-Veronese case, for any $I=(I^1,\dots,I^r)\in\Lambda$, we have
\begin{align}\label{etonthebasis}
e_{I}^{t}
&=\!\!\!\!\sum_{J=(J^1,\dots,J^r)\in \Delta(I)^+}\!\!\!\!
t^{d(I,J)} c_{(I,J)}e_J,
\end{align}
where $ c_{(I,J)}=\binom{s_{I^1}^+}{d(I^1,J^1)}\cdots \binom{s_{I^r}^+}{d(I^r,J^r)}$. 
So we can rewrite the linear subspace $T_t$ as 
\begin{align}\label{Ttonthebasis}
T_t=&\Big< e_I \: | \: d(I,I_0)\leq k_1;\:
\!\!\!\!\sum_{J\in \Delta(I)^+}\!\!\!\!
t^{d(I,J)} c_{(I,J)}e_J\: | \: d(I,I_0)\leq k_2
\Big >.
\end{align}
For future use, we define the set indexing coordinates $z_I$ that do not vanish on some generator of $T_t$:
$$ 
\Delta=\left\{ I \: | \:  d(I,I_0)\leq k_1\right\}\bigcup \left(\bigcup_{d(I,I_0)\leq k_2} \!\!\!\!\Delta(I)^+\right)\subset \Lambda.
$$

On the other hand, by Proposition \ref{oscsegver}, we have 
$$
T^{k_1+k_2+1}_{e_{I_0}}=\left\langle e_I\: | \: d(I,I_0)\leq k_1+k_2+1\right\rangle
=\{z_I=0\: | \: d(I,I_0)> k_1+k_2+1 \}.
$$
In order to prove that $T_0\subset T^{k_1+k_2+1}_{e_{I_0}}$, we will define a family of linear subspaces $L_t$ 
whose flat limit at $t=0$ is $T^{k_1+k_2+1}_{e_{I_0}}$, and such that $T_t\subset L_t$ for every $t\neq 0$. 
(Note that we may assume that  $k_1+k_2\leq d-2$, for otherwise $T^{k_1+k_2+1}_{e_{I_0}}=\mathbb{P}^{N(\pmb{n},\pmb{d})}$.) 
For that, it is enough to 
exhibit, for each pair $(I,J)\in \Lambda^2$ with $d(I,I_0)> k_1+k_2+1$, 
a polynomial $f(t)_{(I,J)}\in \C[t]$ so that the hyperplane $(H_I)_t\subset\mathbb{P}^{N(\pmb{n},\pmb{d})}$ defined by
$$
z_I+t\left( \sum_{J\in \Lambda, \ J\neq I}f(t)_{(I,J)} z_J\right)=0
$$
satisfies $T_t\subset (H_I)_t$ for every $t\neq 0$.
If $I\notin \Delta$, then we can take $f(t)_{(I,J)}\equiv 0$ $\forall J\in \Lambda$. 
So from now on we assume that  $I\in \Delta$. 
We claim that it is enough to  find a hyperplane of type
\begin{align}\label{hyperplane}
F_I=\sum_{J\in \Delta(I)^-}t^{d(I,J)}c_J z_J =0,
\end{align}
with $c_J\in\C$ for $ J\in \Delta(I)^-$, $c_I\neq 0$, and such that $T_t\subset (F_I=0)$ for  $t\neq 0$. 
Indeed, once we find such $F_I$'s, we can take $(H_I)_t$ to be 
$$
z_I+\frac{t}{c_I}\left(\sum_{J\in \Delta(I)^-, \ J\neq I}t^{d(J,I)-1} c_J z_J\right)=0.
$$

In  \eqref{hyperplane}, there are $|\Delta(I)^-|$ indeterminates $c_J$. 
Let us analyze what conditions we get by requiring that $T_t\subseteq (F_I=0)$ for  $t\neq 0$. 
For any $e_K^t$ with non-zero coordinate $z_I$, we have $I\in \Delta(K)^+$, and so $K\in \Delta(I)^-$.
Given $K\in \Delta(I)^-$ we have
\begin{align*}
F_I(e_K^t)&\stackrel{(\ref{etonthebasis})}{=}\
F_I\left(\sum_{J\in \Delta(K)^+}\!\!\!\!
t^{d(K,J)} c_{(K,J)}e_J\right) = \\
&\stackrel{(\ref{hyperplane})}{=}\!\!\!\!\!\!\!\!\!\!\!\!\sum_{J\in \Delta(I)^-\cap \Delta(K)^+}
\!\!\!\!\!\!\!\!\!\!\!\! t^{d(I,K)-d(K,J)}c_J \left(  t^{d(K,J)} c_{(K,J)}\right)
=t^{d(I,K)}\left[\sum_{J\in \Delta(I)^-\cap \Delta(K)^+}\!\!\!\!\!\!\!\! \!\!\!\! c_{(K,J)}c_J \right].
\end{align*}
Thus:
$$
F_I(e_K^t)=0\: \forall\: t\neq 0 \Leftrightarrow\: \sum_{J\in \Delta(I)^-\cap \Delta(K)^+}\!\!\!\!\!\!c_{(K,J)}c_J =0.
$$
This is a linear condition on the coefficients $c_J$, with $J\in \Delta(I)^-$. Therefore
\begin{align}\label{eqns3}
T_t\subset (F_I=0) \mbox{ for } t\neq 0 &\Leftrightarrow 
\begin{cases}
F_I(e_L)=0\   &\forall L\in \Delta(I)^-\cap B[I_0,k_1]  \\ 
F_I(e_K^t)=0 \ \forall t\neq 0 \  &\forall K\in \Delta(I)^-\cap B[I_0,k_2] 
\end{cases}\\&\nonumber
\Leftrightarrow 
\begin{cases}
c_L=0  &\forall L\in \Delta(I)^-\cap B[I_0,k_1]\\
\displaystyle\sum_{J\in \Delta(I)^-\cap \Delta(K)^+}\!\!\!\!\!\!\!\!\!\!\!\!c_{(K,J)}c_J \quad
=0 \ &\forall K\in \Delta(I)^-\cap B[I_0,k_2],
\end{cases}
\end{align}
where $B[J,u]=\{K\in \Lambda |\: d(J,K)\leq u\}$. Set
$$
c=\left|\Delta(I)^-\cap B[I_0,k_1]\right|+\left|\Delta(I)^-\cap B[I_0,k_2]\right|.
$$
The problem is now reduced to finding a solution $(c_J)_{J\in \Delta(I)^-}$ of the linear system given by the $c$ equations (\ref{eqns3}) with $c_I\neq 0$.

In the following we write for short $s=s_I^-$, $\overline{s}=s_I^+$ and $D=d(I,I_0)>k_1+k_2+1$. 
We want to find $s+1$ complex numbers $c_I=c_0,c_1,\dots, c_{s}$ satisfying the following conditions 
\begin{align}\label{eqns4}
\begin{cases}
c_{j}=0  &\forall j=s,\dots,D-k_1\\
\displaystyle\sum_{l=0}^{d(I,K)}\left(c_{d(I,K)-l}\!\!\!\!\displaystyle\sum_{J\in \Delta(I)^-\cap\Delta(K,l)}
\!\!\!\!\!\!\!\!\!\!\!\!c_{(K,J)} \right)=0 
\ &\forall K\in \Delta(I)^-\cap B[I_0,k_2].
\end{cases}
\end{align}
For $0\leq l\leq d(I,K)$, we have
\begin{align*}
\displaystyle\sum_{J\in \Delta(I)^-\cap\Delta(K,l)}
\!\!\!\!\!\!\!\!\!\!\!\!c_{(K,J)}
&=\!\!\!\!\!\!\!\!\displaystyle\sum_{J\in \Delta(I)^-\cap\Delta(K,l)}
\!\!\!\!
\binom{s_{K^1}^+}{d(K^1,J^1)}\cdots \binom{s_{K^r}^+}{d(K^r,J^r)}
=\sum_{\substack{\pmb{l}=(l_1,\dots,l_r) \\ 0\leq l_1,\dots,l_r \\ l_1+\dots+l_r=l}}
\!\!\!\!\binom{s_{K^1}^+}{l_1}\cdots \binom{s_{K^r}^+}{l_r} = \\
&=\binom{s_{K^1}^+ +\dots+s_{K^r}^+}{l}=\binom{s_{K}^+}{l}=\binom{s_{I}^+ + d(I,K)}{l}.
\end{align*}
Thus the system (\ref{eqns4}) can be written as
\begin{align*}
\begin{cases}
c_{j}=0  &\forall j=s,\dots,D-k_1 \\
\displaystyle\sum_{k=0}^{j}\binom{\overline{s} +j}{j-k}c_{k} =0 
\ &\forall j=s,\dots,D-k_2,
\end{cases}
\end{align*}
that is
\begin{align}\label{eqns5}
\begin{cases}
c_{s}=0\\
\vdots\\
c_{D-k_1}=0\\
\end{cases}
\begin{cases}
\binom{\overline{s} +s}{0}c_{s}+\binom{\overline{s} +s}{1}c_{s-1}+\cdots+\binom{\overline{s} +s}{s}c_{0}=0\\
\vdots\\
\binom{\overline{s} +D-k_2}{0}c_{D-k_2}+\binom{\overline{s} +D-k_2}{1}c_{D-k_2-1}+\cdots
+\binom{\overline{s} +D-k_2}{D-k_2}c_{0}=0.
\end{cases}
\end{align}

We will show that the linear system (\ref{eqns5}) admits a solution with $c_0\neq 0$. 
If $s<D-k_2$, then the system (\ref{eqns5}) reduces to $c_s=\dots=c_{D-k_1}=0$. In this case we can take $c_0=1, c_1=\dots,c_s=0$. 
From now on assume that $s\geq D-k_2$. 
Since $c_s=\dots=c_{D-k_1}=0$  in (\ref{eqns5}), we are reduced to checking that the following system admits a solution 
$(c_i)_{0\leq i\leq D-k_1+1}$ with $c_0\neq 0$:
\begin{align*}
\begin{cases}
\binom{\overline{s} +s}{s-(D-k_1+1)}c_{D-k_1+1}+\binom{\overline{s} +s}{s-(D-k_1)}c_{D-k_1}+
\cdots+\binom{\overline{s} +s}{s}c_{0}=0\\
\vdots\\
\binom{\overline{s} +D-k_2}{k_1-1-k_2}c_{D-k_1+1}+\binom{\overline{s} +D-k_2}{k_1-k_2}c_{D-k_1}+
\cdots+\binom{\overline{s} +D-k_2}{D-k_2}c_{0}=0.
\end{cases}
\end{align*}
Therefore, it is enough to check that the $(s-D+k_2+1)\times (D-k_1+1)$ matrix
\begin{align*}
M=
\begin{pmatrix}
\binom{\overline{s} +s}{s-(D-k_1+1)}&\binom{\overline{s} +s}{s-(D-k_1)}&
\cdots&\binom{\overline{s}+s}{s-1}\\
\vdots&\vdots&\ddots &\vdots\\
\binom{\overline{s} +D-k_2}{k_1-1-k_2}&\binom{\overline{s} +D-k_2}{k_1-k_2}&
\cdots&\binom{\overline{s} +D-k_2}{D-k_2-1}
\end{pmatrix}
\end{align*}
has maximal rank. Since $s\leq D$ and $D>k_1+k_2+1$, we have $s-D+k_2+1< D-k_1+1$.
So it is enough to show that the $(s-D+k_2+1)\times (s-D+k_2+1)$ submatrix of $M$
\begin{align*}
M'=&
\begin{pmatrix}
\binom{\overline{s} +s}{s-(s-D+k_2+1)}&\binom{\overline{s} +s}{s-(s-D+k_2)}&
\cdots&\binom{\overline{s} +s}{s-1}\\
\vdots&\vdots& \ddots &\vdots\\
\binom{\overline{s} +D-k_2}{D-k_2-(s-D+k_2+1)}&\binom{\overline{s} +D-k_2}{D-k_2-(s-D+k_2)}&
\cdots&\binom{\overline{s} +D-k_2}{D-k_2-1}
\end{pmatrix} =\\
=&
\begin{pmatrix}
\binom{\overline{s} +s}{\overline{s} +s+1-D+k_2}&\binom{\overline{s} +s}{\overline{s} +s-D+k_2}&
\cdots&\binom{\overline{s} +s}{\overline{s} +1}\\
\vdots&\vdots& \ddots &\vdots\\
\binom{\overline{s} +D-k_2}{\overline{s} +s+1-D+k_2}&\binom{\overline{s} +D-k_2}{\overline{s} +s-D+k_2}&
\cdots&\binom{\overline{s} +D-k_2}{\overline{s}+1}
\end{pmatrix}
\end{align*}
has non-zero determinant. To conclude,  observe that the determinant of $M'$ is equal to the determinant of the matrix of binomial coefficients 
$$
M'':=\left(\binom{i}{j}\right)_{\substack{ \hspace{-0.7cm}
\overline{s} +D-k_2\leq i\leq \overline{s} + s\\ \overline{s} +1\leq j\leq \overline{s} +s+1-D+k_2}}.
$$
Since $D-k_2>k_1+1\geq 1$,  $\det(M')=\det(M'')\neq 0$ by \cite[Corollary 2]{GV85}.
\end{proof}

\begin{Proposition}\label{limitosculatingspacessegverII}
The Segre-Veronese variety $SV_{\pmb{d}}^{\pmb{n}}\subseteq\mathbb{P}^{N(\pmb{n},\pmb{d})}$ has 
$(n_1+1)$-osculating regularity.
\end{Proposition}

\begin{proof}
We follow the same argument and computations as in the proof of Proposition~\ref{limitosculatingspacessegver}.

Given general points $p_0,\dots,p_{n_1}\in SV_{\pmb d}^{\pmb n}\subseteq\mathbb{P}^{N(\pmb n,\pmb d)}$, we may 
apply a projective automorphism of $SV_{\pmb{d}}^{\pmb{n}}\subseteq\mathbb{P}^{N(\pmb{n},\pmb{d})}$ and assume that 
$p_j=e_{I_j}$ for every $j$.
Each $p_j$, $j\geq 1$, is connected to $p_0$ by the degree $d$ rational normal curve defined by 
$$
\gamma_j([t:s])=(se_0+te_j)^{d_1}\otimes\dots\otimes (se_0+te_j)^{d_r}.
$$
We work in the affine chart $(s=1)$, and set $t=(t:1)$. 
Given $k\geq 0$, consider the family of linear spaces 
$$
T_t=\left\langle T^{k}_{p_0},T^{k}_{\gamma_1(t)},\dots,T^{k}_{\gamma_{n_1}(t)}\right\rangle,\: t\in \C\backslash \{0\}.
$$
We will show that the flat limit $T_0$  of $\{T_t\}_{t\in \C\backslash \{0\}}$ in
$\G(\dim(T_t),N(\pmb{n},\pmb{d}))$ is contained in $T^{2k+1}_{p_0}$.

We start by writing the linear spaces $T_t$ explicitly  in terms of the basis $\{e_J|J\in\Lambda\}$. 
As in the proof of Proposition~\ref{limitosculatingspacessegver}, it is convenient to introduce some additional notation. 

Given $I\in \Lambda_{n,d}$, we define $\delta^l_j(I),l\geq 0,$ as in Notation~\ref{defdelta}, 
with the only difference that this time we substitute $0$'s with $j$'s instead of $1$'s.
Similarly, for $I=(I^1,\dots,I^r)\in \Lambda$, $\pmb{l}=(l_1,\dots,l_r)\in \Z^r$, and $l\in \Z$, we define the sets 
$\Delta(I,l)_j, \Delta(I)^+_j, \Delta(I)^-_j \subset \Lambda$, and the integers
$s(I)^+_j,s(I)^-_j\in\{0,\dots,d\}$.

For $j=1,\dots, r$, we define the vectors 
$$
e_0^{j,t}=e_0+te_j,e_1^{j,t}=e_1,e_2^{j,t}=e_2,\dots,e_{n_j}^{j,t}=e_{n_j}\in V_j.
$$
Given $I^l=(i_1,\dots,i_{d_j})\in \Lambda_{n_j,d_j}$, we denote by 
$e_{I^j}^{j,t}\in \Sym^{d_j}V_j$  the symmetric product $e_{i_1}^{j,t}\cdot\ldots\cdot e_{i_{d_j}}^{j,t}$.
Given $I=(I^1,\dots,I^r)\in \Lambda=\Lambda_{\pmb{n},\pmb{d}}$, we denote by 
$e_I^{j,t}\in \mathbb{P}^{N(\pmb{n},\pmb{d})}$
the point corresponding to 
$$
e_{I^1}^{j,t}\otimes\dots\otimes e_{I^r}^{j,t}\in \Sym^{d_1}V_1\otimes\dots\otimes \Sym^{d_r}V_r.
$$
By Proposition \ref{oscsegver} we have 
$$
T_t=\left\langle e_I\: | \: d(I,I_0)\leq k; \: e_I^{j,t}\: |\: d(I,I_0)\leq k, j=1,\dots,n_1\right\rangle,\: t\neq 0.
$$
Now we write each vector $e_I^{j,t}$, with $I=(I^1,\dots,I^r)\in\Lambda$ such that $d(I,I_0)\leq k$, in terms of the basis $\{e_J|J\in\Lambda\}$:
\begin{align*}
e_{I}^{j,t}
&=\!\!\!\!\sum_{J=(J^1,\dots,J^r)\in \Delta(I)^+_j}\!\!\!\!
t^{d(I,J)} c_{(I,J)}e_J
\end{align*}
where $ c_{(I,J)}=\binom{s(I^1)^+_j}{d(I^1,J^1)}\cdots \binom{s(I^r)^+_j}{d(I^r,J^r)}$. 
So we can rewrite the linear subspace $T_t$ as 
\begin{align*}
T_t=&\left\langle e_I \: | \: d(I,I_0)\leq k;\:
\!\!\!\!\sum_{J\in \Delta(I)^+_j}\!\!\!\!
t^{d(I,J)} c_{(I,J)}e_J\: | \: d(I,I_0)\leq k,\ j=1,\dots,n_1
\right \rangle,
\end{align*}
and define the set
$$ 
\Delta=
\bigcup_{1\leq j\leq n_1} 
\bigcup_{d(J,I_0)\leq k} \!\!\!\!\Delta(J)^+_j\subset \Lambda.
$$

On the other hand, by Proposition \ref{oscsegver}, we have 
$$
T^{2k+1}_{p_0}=\left\langle e_I\: | \: d(I,I_0)\leq 2k+1\right\rangle=\{z_I=0\: | \: d(I,I_0)> 2k+1 \}.
$$

As in the proof of Proposition \ref{limitosculatingspacessegver},
in order to prove that $T_0\subset T^{2k+1}_{p_0}$,  it is enough to 
exhibit, for each  $I\in \Delta$ with $d(I,I_0)> 2k+1$, a family of hyperplanes of the form
\begin{equation}\label{eq1}
\left(F_I= \sum_{J\in \Gamma(I)} t^{d(I,J)}c_J z_J=0 \right)
\end{equation}
such that $T_t\subset (F_I=0)$ for $t\neq 0$, and $c_I\neq 0$. Here $\Gamma(I)\subset \Lambda$ is a suitable subset to be defined later. Let $I\in \Delta$ be such that $d(I,I_0)> 2k+1$.
We claim that there is a unique $j$ such that 
\begin{equation}\label{I->j}
I\in \bigcup_{d(J,I_0)\leq k} \!\!\!\!\Delta(J)^+_j.
\end{equation}
Indeed, assume that $I\in \Delta(J,l)_i$ and $I\in\Delta(K,m)_j$, with $d(J,I_0),d(K,I_0)\leq k$.
If $i\neq j$, then we must have 
$$
 d(J,I_0) \geq m \ \ \text{ and } \ \  d(K,I_0) \geq l.
$$
But then $d(I,I_0) = d(J,I_0) + l \leq d(J,I_0) + d(K,I_0)  \leq 2k$, contradicting the assumption that $d(I,I_0)> 2k+1$.
Let $J$ and $j$ be such that $d(J,I_0)\leq k$ and $I\in \Delta(J)^+_j$.
Note that $d(I,I_0)-s(I)^-_j=d(J,I_0) - s(J)^-_j\leq k$, and hence $k+1-d(I,I_0)+s(I)^-_j>0$.
We set $D=d(I,I_0)$ and define 
\begin{equation}\label{Gamma}
\Gamma(I)=\!\!\!\!\!\!\!\!\bigcup_{0\leq l \leq k+1-D+s(I)^-_j} \!\!\!\!\!\!\!\!\Delta(I,-l)_j \subset \Lambda.
\end{equation}
This is the set to be used in \eqref{eq1}. First we claim that 
\begin{equation}\label{eq2}
J\in \Gamma(I)\Rightarrow J \notin  \bigcup_{\substack{1\leq i\leq n_1 \\ i\neq j} }
\bigcup_{d(I,I_0)\leq k} \!\!\!\!\Delta(I)^+_i.
\end{equation}
Indeed, assume that $J\in \Delta(I,-l)_j$ with $0\leq l \leq k+1-D+s(I)^-_j$, and
$J\in \Delta(K)^+_i$  for some $K$ with $d(K,I_0)\leq k$.
If $i\neq j$, then 
$$
s(K)^-_j=s(J)^-_j=s(I)_j^- - l\geq D-(k+1)>k,
$$
contradicting the assumption that $d(K,I_0)\leq k$.
Therefore, if $F_I$ is as in \eqref{eq1} with $\Gamma(I)$ as in \eqref{Gamma}, then we have 
$$
\left\langle e_I\: | \: d(I,I_0)\leq k; \: \sum_{J\in \Delta(I)^+_i}\!\!\!\!
t^{d(I,J)} c_{(I,J)}e_J\: | \: d(I,I_0)\leq k,\ i=1,\dots,n_1, i\neq j \right\rangle \subset (F_I=0) , t\neq 0,$$
and thus
$$
T_t\subset (F_I=0), t \neq 0 \Longleftrightarrow \left\langle \sum_{J\in \Delta(I)^+_j}\!\!\!\!
t^{d(I,J)} c_{(I,J)}e_J\: | \: d(I,I_0)\leq k \right\rangle \subset (F_I=0) , t\neq 0.
$$

The same computations as in the  proof of Proposition \ref{limitosculatingspacessegver} yield
\begin{equation}
\label{eq3}T_t\subset (F_I=0), t \neq 0 \Longleftrightarrow \sum_{J\in \Delta(K)^+_j \cap \Gamma(I)}\!\!\!\! c_{J}c_{(K,J)}=0
\ \ \forall  K\in \Delta(I)^-_j\cap B[I_0,k].
\end{equation}
So the problem is reduced to finding a solution $(c_J)_{J\in\Gamma(I)}$ for the linear system \eqref{eq3} such that $c_I\neq 0$. 
We set $c_J=c_{d(I,J)}$ and reduce, as in the proof of Proposition \ref{limitosculatingspacessegver}, to the linear system
\begin{equation}\label{linsys}
\displaystyle\sum_{l=0}^{k+1-D+s(I)^-_j}\binom{d-i}{D-l-i}c_{l} =0 ,
\ \ \ D-s(I)^-_j\leq i \leq k
\end{equation}
in the variables $c_0,\dots,c_{k+1-D+s(I)^-_j}$. 
The argument used in the end of Proposition \ref{limitosculatingspacessegver} shows that the linear system (\ref{linsys}) admits a solution with $c_0\neq 0$.
\end{proof}

\end{document}